\documentclass[a4paper, 12pt, oneside, reqno]{amsart}

\usepackage[top=35mm,bottom=15mm,left=20mm,right=20mm]{geometry}

\usepackage[foot]{amsaddr} 
\usepackage{mathrsfs}
\usepackage{amsfonts}
\usepackage{amssymb}
\usepackage{amsxtra}
\usepackage{color}
\usepackage{dsfont}
\usepackage{bbm}

\usepackage[english,polish]{babel}
\usepackage[colorlinks,citecolor=blue,urlcolor=blue,bookmarks=true]{hyperref}
\usepackage{mathtools}
\usepackage[disable]{todonotes}

\newcommand{\N}{\mathbb{N}}
\newcommand{\R}{\mathbb{R}}

\newcommand{\ud}{\mathrm{d}}
\newcommand{\Rd}{\R^d}

\newcommand{\calR}{\mathcal{R}}

\newcommand{\pl}[1]{\foreignlanguage{polish}{#1}}

\newtheorem{theorem}{Theorem}
\newtheorem{proposition}[theorem]{Proposition}
\newtheorem{lemma}[theorem]{Lemma}

\theoremstyle{definition}

\title[Alexandrov Theorem for nonlocal curvatur]{Alexandrov Theorem for nonlocal curvature}
\author{Wojciech Cygan $^{1}$}
\address{$^1$University of Wroc{\l}aw,
		Faculty of Mathematics and Computer Science\\
		Institute of Mathematics,
		pl.\ Grunwaldzki 2/4, 50--384 Wroc{\l}aw, Poland}
\email{wojciech.cygan@uwr.edu.pl}
\thanks{Research supported by National Science Centre (Poland), grant no.\ 2019/33/B/ST1/02494}

\author{Tomasz Grzywny $^{2}$}
\address{$^{2}$Wroc{\l}aw University of Science and Technology,
Faculty of Pure and Applied Mathematics\\
	Wyb. \pl{Wyspia\'{n}skiego} 27,
	50-370 \pl{Wroc\l{}aw}, Poland}
\email{tomasz.grzywny@pwr.edu.pl}

 				

\numberwithin{equation}{section}
\allowdisplaybreaks

\begin{document}
\selectlanguage{english}

\begin{abstract}
In this article we obtain a nonlocal version of the Alexandrov Theorem which asserts that the only set with sufficiently smooth boundary and of constant nonlocal mean curvature is an Euclidean ball. 
We consider a general nonlocal mean curvature given by a radial and monotone kernel 
and we formulate an easy-to-check condition which is necessary and sufficient for the nonlocal version of the  Alexandrov Theorem to hold in the treated context. Our definition encompasses numerous examples of various nonlocal mean curvatures that have been already studied in the literature.
To prove the main result we obtain a specific formula for the tangential derivative of the nonlocal mean curvature and combine it with an application of the method of moving planes. 
\end{abstract}

\maketitle

\section{Introduction}
In the present article we consider a general version of isotropic nonlocal mean curvature and we study the corresponding constant-curvature problem in the spirit of the classical Alexandrov Theorem which holds for  hypersurfaces in the Euclidean space. More precisely, 
let $j\colon \Rd \to [0,\infty)$ be a given radial and radially non-increasing function
 satisfying
\begin{equation}\label{eq:levy}
\int_{|x|<1}  |x|^{\beta-1} \, j(x) \, \ud x + \int_{|x|\geq 1}j(x)\, \ud x < \infty,
\end{equation}
where $\beta \in (0,2]$ corresponds to the regularity of boundaries of sets that we shall allow in our definition of the mean curvature. 
For any $\gamma>0$ we introduce the following shorthand-notation
\begin{align*}
C^\gamma := \mathcal{C}^{\lceil\gamma \rceil -1,\, \gamma +1 - \lceil \gamma \rceil} .
\end{align*}
Note that if $\beta =2$ then $C^\beta$ corresponds to $\mathcal{C}^{1,1}$ and $C^{1+\beta}$ is simply $\mathcal{C}^{2,1}$.
For any set $E \subset \Rd$ with boundary of class $C^{\beta}$ for some  $\beta \in (0,2]$, the nonlocal mean curvature of $\partial E$ at the point $x \in \partial E$ corresponding to the kernel $j$ is defined as
\begin{equation}\label{eq:mean}
H_{j} (x,E) =   \lim_{r\downarrow 0}  \int_{B_r(x)^c} \widetilde{\chi}_E(y) j(x-y) \, \ud y,
\end{equation}
where 
\begin{align}\label{eq:diff-indic}
\widetilde{\chi}_{E}(x) = \chi_{E^c} (x) - \chi_E (x),
\end{align}
and $\chi_E (x)$ is the characteristic function of the set $E\subset \R^d$.
The integral in \eqref{eq:mean} is defined in the principal value sense and one can show that because of cancellations this object exists and is finite, cf.\ \cite{CGL-2023}. Moreover, in the same work it was shown that the mean curvature $H_j$ is the first variation of the corresponding $j$-perimeter. Such perimeters were  investigated in \cite{CG-Annali} and it was proved that \eqref{eq:j-tail} for $\beta=2$ is a necessary condition under which all bounded sets of finite (classical) perimeter have finite $j$-perimeter. On the other hand, condition \eqref{eq:j-tail} for $\beta <2$ enables us to apply our results for sets with boundaries of lower regularity. 

The main objective of the present article is to study a nonlocal version of the famous Alexandrov Theorem \cite{Alexandrov} which asserts that the only closed embedded hypersurface in $\R^d$ of constant mean curvature is a sphere. 
 Our main contribution is the following theorem which provides us with a corresponding result for the nonlocal curvature given in \eqref{eq:mean}.

\begin{theorem}\label{thm:main}
Let $E$ be a connected, open and bounded subset of $\R^d$ such that $\partial E\in C^{1+\beta}$ for some $\beta \in (0,2]$. Suppose that the kernel $j$ satisfies \eqref{eq:levy} and
\begin{align}\label{eq-j-at-zero}
\lim_{r\downarrow 0}j(r) >j(\rho),\quad \rho>0.
\end{align}
 If the nonlocal mean curvature $H_j(x,E)$ is constant along the boundary $\partial E$, then $E$ must be an Euclidean ball. 
\end{theorem}

We  stress that condition \eqref{eq-j-at-zero} is optimal in the present setting in the sense that if \eqref{eq-j-at-zero} does not hold, then there exists an open, connected and bounded set $E$ such that $H_j(\cdot, E)$ is constant along $\partial E$, but $E$ is not a ball, see Proposition \ref{prop:counter}.

We point out that 
the assumption about connectedness of $E$ can be removed if the support $\mathrm{{supp}(j)}$ is equal to the whole of $\R^d$. On the other hand, 
if the kernel $j$ has compact support we could apply the same technique as in the proof of \cite[Theorem 1.2]{Biswas} 
to show that if $E$ is not connected and has constant nonlocal mean curvature, then it is a union of Euclidean balls of the same radius that are separated from each other and the distance between them depends on the size of $\mathrm{supp} (j)$.

The first attempt to show a nonlocal version of the Alexandrov Theorem was undertaken in \cite{Cabre} where the authors investigated  the case of the fractional mean curvature, that is the mean curvature governed by the singular kernel $j(x) = |x|^{-d-\alpha}$, where $\alpha \in (0,1)$ and $|x|$ denotes the Euclidean norm of $x\in \R^d$. The authors of \cite{Cabre} established the desired result for sets with boundaries of class $C^{2,\gamma}$, for $\gamma >\alpha$.
Similar result was found independently  (obtained through slightly different methods and for $C^{1,\gamma}$ boundaries) in \cite{Ciraolo}, where the authors proved also stability results, see also \cite{Cabre-2018} and \cite{Cabre-Math-Ann} for other interesting and related problems. 
The same question but for nonlocal mean curvatures governed by integrable kernels has been recently studied in \cite{Biswas}. Apart from integrability of the kernel $j$, the authors additionally assumed differentiability and  pointwise estimates for the function $j$ and its gradient. 

Another attempt to generalize the nonlocal Alexandrov Theorem for logarithmic kernels was undertaken in \cite{Log-case}. Unfortunately, the main result in this direction, that is \cite[Theorem 4.14]{Log-case}, is evidently not true as one can easily check that the set consisting of two balls of the same radius and separated from each other by distance at least one, has constant $log$-curvature and it is not a ball. 
Our Theorem \ref{thm:main} allows not only logarithmic kernels but it also covers the known versions of the  Alexandrov Theorem for the fractional mean curvature from  \cite{Cabre}  and for the integrable kernel cases studied in \cite{Biswas}. The reason is that we consider  general kernels with no further integrability or differentiability conditions. For instance, our result is valid for any unbounded kernels $j$ satisfying \eqref{eq:levy}.
One possible interesting choice for a family of admissible kernels that illustrates the range of applicability of Theorem \ref{thm:main} is the following bunch of functions
\begin{align*}
j(x) = \sum_{n\geq 1}b_n \chi_{(0,a_n)}(|x|),\quad x\in \R^d,
\end{align*}
where $\mathrm{supp}(j) = (0,a_1]$ and $a_n$ is any sequence of positive numbers converging to zero while $b_n >0$ is chosen to secure that $\sum_{n\geq 1}b_n a_n^{\beta +1}<\infty$, which coincide with our integrability condition from \eqref{eq:levy}. 

A strongly related framework of $j$-criticality has been developed in \cite{Bucur} where the authors obtained a version of the Alexandrov Theorem for locally integrable kernels satisfying some improved integrability conditions which enabled them to weaken the assumptions on the regularity of the boundary, see also \cite{Bucur-AdvMath} for a similar context. 

Our main motivation for the study of nonlocal mean curvature comes from the fact that, similarly as in the differential geometry, the classical mean curvature of a surface plays a key role in the study of minimal surfaces  (see e.g.\ \cite{Giusti}), the fractional mean curvature turned out to be an appropriate tool to find a minimizer determined via the Euler-Lagrange equation for the fractional Laplacian, as discovered in the seminal paper \cite{Caffarelli_Savin}, see also \cite{Imbert}. On the other hand, fractional mean curvature can approximate the classical mean curvature if one applies a specific scaling procedure; this approach was initiated in \cite{Abatan-Valdin-curv} and refined in various directions by many authors, see \cite{Rossi_paper_1}, \cite{Rossi_book}, \cite{Chambolle-Archiv}, \cite{Chambolle-Interfaces}, \cite{Valdinoci-Flattering}, \cite{MR3824212},  \cite{CGL-2023}. 

For the proof of our main result we adopt the technique of moving planes which originates from \cite{Alexandrov}. The method of moving planes has its roots in PDEs and it has numerous applications in the nonlocal setting. For instance, it was successfully applied to solve fractional overdetermined problems, see e.g.\ \cite{Fall} and \cite{Dalibard}.
We combine this method with an application of a specific formula for the tangential derivative of the mean curvature; see \eqref{tang-deriv-H}. It is worth pointing out that in view of the lack of differentiability assumption, such a formula is an interesting and novel contribution on its own. 
Similar approach was also used in \cite{Cabre} and \cite{Ciraolo} but for a special case of the fractional mean curvature and the proof was heavily based on the precise form of the fractional kernel. This technique was also applied in \cite{Biswas} for integrable kernels but the authors exploited the differentiability and strict monotonicity assumptions. 

\section{Tangential derivative formula}

The main aim in this section  is to prove a formula for the tangential derivative of the nonlocal mean curvature which will be applied later to prove Theorem \ref{thm:main}. 

We slightly abuse the notation by writing $j(r)$ for $j(x)$ whenever $|x| = r$, for any $r>0$.  Since $j$ is  monotone, without loss of generality we may assume that $j$ is right-continuous and whence 
it can be  represented uniquely by the tail of a positive measure $\nu$, that is we have  
\begin{align}\label{eq:j-tail}
j(r) = \int_r^\infty \nu(\ud s),\quad r>0.
 \end{align}
We also set $\nu (\{0\})=0$.
We start with the following auxilliary lemma.
\begin{lemma}\label{lem:tang-deriv-help}
Suppose that $E$ is open, bounded set such that $\partial E \in C^{1+\beta}$, for some $\beta \in (0,2]$, and assume that the kernel $j$ satisfies \eqref{eq:levy}. 
For any fixed $\varepsilon >0$ we consider a function $\partial E \ni x\mapsto H_\varepsilon (x,E)$ defined as
\begin{align*}
H_\varepsilon (x,E) = \int_{|z|\geq \varepsilon} \widetilde{\chi}_E(x+z)j(|z|)\ud z.
\end{align*}
Then $H_\varepsilon (\cdot,\, E)\in \mathcal{C}^1(\partial E)$ and, for each $i=1,\ldots ,d$, it holds
\begin{align}\label{H_eps-defiv}
\partial_{x_i}H_\varepsilon (x,E) 
=
\int_{\varepsilon}^\infty \!\!\! s^{-1}\! \int_{|z|=s}\widetilde{\chi}_E (x+z)z_i\, \sigma (\ud z)\, \nu (\ud s)
+
\varepsilon^{-1}
j(\varepsilon) \int_{|z|=\varepsilon}\widetilde{\chi}_E(x+z) z_i\, \sigma (\ud z).
\end{align}
\end{lemma}
\begin{proof}
Let $n\in \N$ and let $\varphi_n$ be a sequence of smooth functions with compact supports and such that $|\varphi_n|\leq 1$ and 
\begin{align}\label{chi-approx}
\varphi_n \to \widetilde{\chi}_E,\quad \text{uniformly in}\ \R^d\setminus \partial E.
\end{align}
Further, let 
\begin{align*}
H_\varepsilon^n(x,E) = \int_{|z|\geq \varepsilon}\varphi_n(x+z)j(|z|)\, \ud z.
\end{align*}
In view of the dominated convergence theorem this functions are of class $\mathcal{C}^1$ and \eqref{eq:j-tail} combined with the Fubini theorem and the integration by parts formula yield
\begin{align*}
\partial_{x_i}H_\varepsilon^n(x,E) 
&= 
\int_{|z|\geq \varepsilon}\int_{|z|}^\infty \partial_{x_i}\varphi_n (x+z)\, \nu (\ud s)\, \ud z
=\int_{\varepsilon}^\infty \int_{\varepsilon \leq |z|\leq s}\partial_{x_i}\varphi_n (x+z)\, \ud z\, \nu (\ud s)\\
&=
\int_{\varepsilon}^\infty \!\! s^{-1} \int_{|z|=s}\varphi_n(x+z) z_i\, \sigma(\ud z)\, \nu (\ud s)
+
\varepsilon^{-1}j(\varepsilon)\int_{|z|=\varepsilon}\varphi_n (x+z)z_i\, \sigma (\ud z).
\end{align*}
We next observe that \eqref{chi-approx} implies that, for any $x\in \partial E$,
\begin{align*}
|\varphi_n(y) - \widetilde{\chi}_E(y)|
\to 
0,\quad 
\text{uniformly in}\ \{y\in \R^d: |y-x|\geq \varepsilon \},\ \text{as}\ n\to \infty.
\end{align*}
This yields \todo[size=\tiny]{more details}
\begin{align*}
\lim_{n\to \infty} 
\partial_{x_i}H_\varepsilon^n(x,E) 
&= 
\int_{\varepsilon}^\infty \!\! s^{-1} \int_{|y-x|=s}\widetilde{\chi}_E(y) (y_i-x_i)\, \sigma(\ud y)\, \nu (\ud s)\\
&\hspace{1cm}
+
\varepsilon^{-1}j(\varepsilon)\int_{|y-x|=\varepsilon}\widetilde{\chi}_E (y) (y_i-x_i)\, \sigma (\ud y),
\end{align*}
and this evidently completes the proof. 
\end{proof}

\begin{proposition}\label{thm:tangential}
Let $E$ be an open and  bounded set with boundary of class $C^{1+\beta}$, for some $\beta \in (0, 2]$. Suppose that the kernel $j$ satisfies \eqref{eq:levy}. 
Let $T_x \partial E$ denote the tangent space to $\partial E$ at $x\in \partial E$. Then, $H_j(x,E)\in \mathcal{C}^1(\partial E)$ and, for any $v\in T_x\partial E$,
\begin{align}\label{tang-deriv-H}
\partial_v H_j(x,E) 
=
\lim_{\varepsilon \downarrow 0}
\int_{\varepsilon}^\infty \!\! s^{-1} \int_{|x-y|=s}\widetilde{\chi}_E(y) (x-y)\cdot v\, \sigma(\ud y)\, \nu (\ud s).
\end{align}
\end{proposition}

\begin{proof}
Using the notation from Lemma \ref{lem:tang-deriv-help}, by \eqref{H_eps-defiv} we have
\begin{align*}
\partial_v H_\varepsilon (x,E)
&=
\int_{\varepsilon}^\infty \!\!\! s^{-1}\! \int_{|z|=s}\widetilde{\chi}_E (x+z)z\cdot v\, \sigma (\ud z)\, \nu (\ud s)
+
\varepsilon^{-1}
j(\varepsilon) \int_{|z|=\varepsilon}\widetilde{\chi}_E(x+z) z\cdot v\, \sigma (\ud z).
\end{align*}
For any $s>0$ we set
\begin{align*}
I_{x,v}(s)
=
\int_{|z|=s}\widetilde{\chi}_E(x+z) z\cdot v\, \sigma (\ud z).
\end{align*}
We next take any $0<\varepsilon <\varepsilon'$ and observe that
\begin{align*}
\left\vert \partial_v H_\varepsilon (x,E) - \partial_v H_{\varepsilon'} (x,E)\right\vert
\leq
\int_{\varepsilon}^{\varepsilon'}\!\!\! s^{-1}\! \left\vert I_{{x,\nu}}(s)\right\vert \, \nu (\ud s)
+
\varepsilon^{-1}
j(\varepsilon) \left\vert I_{x,v}(\varepsilon)\right\vert+
(\varepsilon')^{-1}
j(\varepsilon')
|I_{x,v}(\varepsilon') |.
\end{align*}
Using the same reasoning as in \cite[Lemma 3.2]{Cabre}, since $\partial E\in C^{1+\beta}$, one can show that
\begin{align}\label{Cabre-bound}
|I_{x,v}(\varepsilon)|\leq C |v| \varepsilon^{d+\beta}.
\end{align}
This implies that
\begin{align*}
\left\vert \partial_v H_\varepsilon (x,E) - \partial_v H_{\varepsilon'} (x,E)\right\vert
&\leq
C|v|\left(
\int_{\varepsilon}^{\varepsilon'} s^{d+\beta -1}\nu (\ud s)+ \varepsilon^{d+\beta -1}j(\varepsilon) +
(\varepsilon')^{d+\beta -1}j(\varepsilon') \right) .
\end{align*}
Finally,
\begin{align*}
\int_{\varepsilon}^{\varepsilon'} s^{d+\beta -1}\nu (\ud s)
&=
(d+\beta -1)
\int_\varepsilon^{\varepsilon'}\int_0^s u^{d+\beta -2}\ud u\,  \nu (\ud s)
=
(d+\beta -1)
\int_{0}^{\varepsilon'}u^{d+\beta -2}\int_{u\vee \varepsilon}^{\varepsilon'}\nu(ds)\, \ud u\\
&\leq 
\varepsilon^{d+\beta -1}j(\varepsilon) + 
(d+\beta -1)\int_{\varepsilon}^{\varepsilon'}u^{d+\beta -2}\int_{u}^{\varepsilon'} \nu (\ud s)\, \ud u\\
&\leq 
\varepsilon^{d+\beta -1}j(\varepsilon) + 
(d+\beta -1)\int_{\varepsilon}^{\varepsilon'}u^{d+\beta -2}j(u)\, \ud u\\
&\leq 
C \int_0^{\varepsilon'}u^{d+\beta -2}j(u)\, \ud u \to 0,\quad \text{as}\ \varepsilon'\downarrow 0,
\end{align*}
where we used the fact that $j$ is non-increasing and combined it with \eqref{eq:levy}. \todo[size=\tiny]{new\\ assumption?}
If we now choose any sequence $\varepsilon_k\downarrow 0$, then, by applying the same reasoning as in the proof of Lemma 3.2 and Lemma 3.3 from \cite{CGL-2023},  we can conclude that 
$H_{\varepsilon_k}(x,E)\to H_j(x,E)$ in $\mathcal{C}(\partial E)$. Then
the estimates presented above imply that $\{H_{\varepsilon_k}(x,E)\}$ forms an uniform Cauchy sequence in $\mathcal{C}^{1}(\partial E)$ and whence, independently of the choice of the sequence $\{\varepsilon_k\}$,
\begin{align*}
\lim_{k\to \infty} \partial_{v} H_{\varepsilon_k}(x,E) = 
\lim_{\varepsilon \downarrow 0}
\int_{\varepsilon}^\infty \!\! s^{-1} \int_{|x-y|=s}\widetilde{\chi}_E(y) (x-y)\cdot v\, \sigma(\ud y)\, \nu (\ud s)
\end{align*}
and this completes the proof. 
\end{proof}

\section{Nonlocal Alexandrov Theorem}

In this section we focus on the proof of Theorem \ref{thm:main}. Before we embark on the proof, we introduce the necessary notation and the main objects needed for the application of the moving planes method. 

We fix an arbitrary direction $e\in \mathbb{S}^{d-1}$ and we consider the family of hyperplanes orthogonal to $e$ which are denoted by
\begin{align*}
\pi_\lambda = \{ x\in \R^d:\, x\cdot e = \lambda \},\quad \lambda \in \R.
\end{align*}
The positive \textit{cap} of $E$ related to the cut by the hyperplane $\pi_\lambda $ is the set
\begin{align*}
E_\lambda = E\cap \{x\in \R^d:\, x\cdot e>\lambda \}.
\end{align*}
Next, for any $\lambda$, the linear operator describing the reflection with respect to the hyperplane $\pi_\lambda$ will be denoted by $\calR_\lambda$. We note that
\begin{align*}
\calR_\lambda (x) = x-2(x\cdot e - \lambda)e,\quad x\in \R^d.
\end{align*}
The reflection of $E$ will be denoted by $\calR_\lambda (E)$. Since $E$ is bounded, we can define 
\begin{align*}
\mathfrak{s}= \max_{x\in E}\, x\cdot e,
\end{align*}
and it follows that $\pi_\lambda \cap E = \emptyset$, for $\lambda >\mathfrak{s}$. Moreover, as the boundary $\partial E$ is of class $\mathcal{C}^{1}$, for any $\lambda <\mathfrak{s}$ but close enough to $\mathfrak{s}$, the positive cap $E_\lambda$ will be close to the boundary and thus its reflection $\mathcal{R_\lambda }(E_\lambda)$ will be contained in $E$. We consider the smallest number satisfying this property, that is, let
\begin{align*}
\lambda_\ast = \inf \{\lambda\in \R:\, \calR_s (E_s)\subset E,\, \text{for all}\ s \in (\lambda,\mathfrak{s})\}.
\end{align*}
The corresponding hyperplane $\pi_{\lambda_\ast}$ is called the critical hyperplane and it will be denoted simply by $\pi_\ast$. We also write $\pi_{-} = \{x\in \R^d :\, x\cdot e< \lambda_\ast \}$.
 The associated reflection operator $\calR_{\lambda_\ast}$ will be denoted by $\calR_\ast$, and we will write $E_\ast$ for the reflection $\calR_\ast (E)$.
If we continuously  decrease $\lambda$ until it reaches $\lambda_\ast$, one of the two following geometric situations will happen:
\begin{itemize}
\item[Case 1.] The boundary of $E_\ast$ will touch the boundary $\partial E$ at a certain point $p$ lying in the interior of $\partial \calR_\ast (E_{\lambda_\ast})$ and its reflection  $p_\ast=\calR_\ast(p) \in \partial E_{\lambda_\ast} \cap \pi_\ast^c$. In this case the point $p$ is called a \textit{touching point};
\item[Case 2.] There is a point $q\in \partial E\cap \pi_\ast$ at which $\pi_\ast$ is orthogonal to $\partial E$. In this situation the point $q$ is called a point of \textit{non-transversal intersection}.
\end{itemize}
To establish that $\partial E$ is a sphere it is enough to show that $E_\ast = E$ as it would mean that for any fixed direction there exists a hyperplane of symmetry for $\partial E$. The details can be found in \cite{Ciraolo}.

\begin{proof}[Proof of Theorem \ref{thm:main}]
In the proof we will show that $E_\ast= E$ in each of the two cases that we considered above. 

\medskip
\textit{Case 1}. Without loss of generality we can assume that $e=(1,0,\ldots ,0)\in \mathbb{S}^{d-1}$ which yields that $x_\ast = \calR_\ast (x) = (2\lambda -x_1,x')$, where $x' = (x_2,\ldots ,x_d)$. We first observe that, by the radial symmetry of $j$ and a change of variable we easily obtain that
\begin{align*}
H_j(p_\ast ,E) = H_j (p, E_\ast).
\end{align*}
Further, we have
\begin{align*}
\widetilde{\chi}_E(x) = \widetilde{\chi}_{E_\ast} (x), \quad \text{for}\  x\in \left( E\cap E_\ast\right) \cup \left( E^c\cap  E_\ast^c\right),
\end{align*}
and similarly,
\begin{align*}
-\widetilde{\chi}_E(x) = \widetilde{\chi}_{E_\ast} (x) =1 , \quad \text{for}\ x\in  E\cap E_\ast^c ,
\end{align*}
and 
\begin{align*}
\widetilde{\chi}_E(x) = -\widetilde{\chi}_{E_\ast} (x) =1 , \quad \text{for}\ x\in  E^c\cap E_\ast. 
\end{align*}
Let
\begin{align*}
A_\varepsilon^1 = \{y\in E\setminus E_\ast :\, |p-y|\geq \varepsilon\},\quad \text{and}\quad 
A_\varepsilon^2 =\{ y\in E_\ast \setminus E:\, |p-y|\geq \varepsilon\}.
\end{align*}
Since $H_j(x,E)$ is constant along $\partial E$, we have
\begin{align}\label{eq:correct1}
0 &= H_j(p , E) - H_j(p_\ast ,E) = H_j(p, E) - H_j(p, E_\ast) \nonumber \\
&=
\lim_{\varepsilon \downarrow 0}\left( \int_{A_\varepsilon^1} + \int_{A_\varepsilon^2}\right) \left( \widetilde{\chi}_E(y) - \widetilde{\chi}_{E_\ast}(y)\right) j(|p -y|)\, \ud y \nonumber \\
&=
2\lim_{\varepsilon \downarrow 0}\left\{ \int_{A_\varepsilon^1}j(|p -y|) \, \ud y
- \int_{A_\varepsilon^2}j(|p -y|) \, \ud y \right\} \nonumber \\
&=
2\lim_{\varepsilon \downarrow 0}\left\{ \int_{A_\varepsilon^1}j(|p -y|) \, \ud y
- \int_{A_\varepsilon^1}j(|p -y_\ast|) \, \ud y \right\}.
\end{align}
Since $y_1\geq \lambda_\ast \geq (y_\ast)_1$ and $p_1<\lambda$, we conclude that $|p - y|< |p - y_\ast|$ in $E\setminus E_\ast$, which in view of the monotonicity of $j$ implies that 
the integrand in \eqref{eq:correct1} is non-negative.
Hence, we arrive at 
\begin{align*}
0 =
\int_{E \setminus E_\ast} \left\{ j(|p -y|) - j(|p - y_\ast|)\right\} \ud y.
\end{align*}
Set $\varrho = \mathrm{dist}(p,\pi_\ast)/2$. It follows that
\begin{align}\label{eq:j-diff-proof}
0 = 
\int_{E \setminus E_\ast} \left\{ j(|p -y|) - j(|p- y_\ast|)\right\} \ud y
&\geq 
\int_{(E \setminus E_\ast )\cap B_\varrho (p)} \left\{ j(|p -y|) - j(|p - y_\ast|)\right\} \ud y \nonumber\\
&\geq \int_{(E \setminus E_\ast )\cap B_\varrho (p)} \left\{ j(|p -y|) - j(2\varrho)\right\} \ud y \nonumber\\
&\geq \int_{(E \setminus E_\ast )\cap B_{\varrho_0} (p)} \left\{ j(|p -y|) - j(2\varrho)\right\} \ud y,
\end{align} 
where we choose $\varrho_0>0$ in such a way that $j(\varrho_0)>j(2\varrho)$, which is possible according to \eqref{eq-j-at-zero}. Together with the monotonicity of $j$ this implies that the integrand in \eqref{eq:j-diff-proof} is strictly positive and whence it must hold
$|(E \setminus E_\ast)\cap B_{\varrho_0}(p)| = 0$. We thus conclude that the set of all interior touching points in $\partial E_{-}$ is an open set. One can very easily show that this set is closed at the same time, and whence each point in $\partial E_{-}$ must be a point of interior touching, which means that $E=E_\ast$.

\textit{Case 2}. Suppose that $q$ is a point of non-transversal intersection. 
In view of the proof of Case 1, without loss of generality we may and do assume that there are no touching points. We shall aim for a contradiction. 
We again may choose $e=e_1$ and this implies that $q_1=\lambda_\ast$, as $q\in \pi_\ast \cap \partial E$. Since $H_j(\cdot ,E)$ is constant over $\partial E$, according to Proposition \ref{thm:tangential} we obtain that
\begin{align}\label{eq:Case-2}
\begin{split}
0
&=
\partial_{e_1}H_j(q,E) - \partial_{e_1}H_j(q,E_\ast)\\
&=
\lim_{\varepsilon \downarrow 0}
\int_{\varepsilon}^\infty s^{-1}
\int_{|q-y|=s}\left( \widetilde{\chi}_E (y) - \widetilde{\chi}_{E_\ast}\right) (\lambda_\ast - y_1)\, \sigma(\ud y)\, \nu(\ud s)\\
&=
\lim_{\varepsilon \downarrow 0}
\int_{\varepsilon}^\infty s^{-1}
\Bigg\{ 
\int\limits_{
             \begin{subarray}{c}
        |q-y|=s\\
          y\in E\setminus E_\ast
             \end{subarray}}
             (\lambda_\ast -y_1)\, \sigma(\ud y)
             -
    \int\limits_{
             \begin{subarray}{c}
        |q-y|=s\\
          y\in E_\ast \setminus E
             \end{subarray}}
             (\lambda_\ast -y_1)\, \sigma(\ud y)         
             \Bigg\}\, 
             \nu(\ud s)\\ 
 &=
\int_{0}^\infty s^{-1}
\Bigg\{ 
\int\limits_{
             \begin{subarray}{c}
        |q-y|=s\\
          y\in E\setminus E_\ast
             \end{subarray}}
             |\lambda_\ast -y_1|\, \sigma(\ud y)
             +
    \int\limits_{
             \begin{subarray}{c}
        |q-y|=s\\
          y\in E_\ast \setminus E
             \end{subarray}}
             |\lambda_\ast -y_1|\, \sigma(\ud y)         
             \Bigg\}\, 
             \nu(\ud s).
 \end{split}
\end{align}
The last equality is justified as $y_1<\lambda_\ast$ for $y\in E\setminus E_\ast $, and $y_1>\lambda_\ast$, for $y\in E_\ast\setminus E$, and these two inequalities imply that
$\mu (E \triangle E_\ast) = 0$, where $\mu$ is a positive measure given by
\begin{align*}
\mu (A) = \int_0^\infty \int_{\mathbb{S}^{d-1}(q)} \chi_A (r\theta) \, \sigma (\ud \theta)\, \nu (\ud r),\quad A\in \mathcal{B}(\R^d).
\end{align*}
We recall that $E$ and $E^\ast$ have both $C^{1+\beta}$ boundaries and since $q$ is a point of non-transversal touching, the two normal vectors for $E$ and $E_\ast$ at $q$ must coincide. Hence, there exists a coordinate system such that in a ball $B_R(q)$ the set $E$  is a sub-graph of a given function $f$ and, in the same ball, $E^\ast$ is a sub-graph of a function $f_\ast$, where $R>0$ is the localisation radius. Since there are no touching points, it must hold that $f>f_\ast$ on a part of their domains contained in $\pi_{-}$. 
By continuity of these two functions, we conclude that $\sigma \left(\mathbb{S}^{d-1}_r(q))\cap (E\triangle E_\ast)\right)>0$, for $r\in(0, R_0)$, for some $R_0\leq R$. Hence, in view of \eqref{eq:Case-2}, it must hold $\nu (0 ,R_0)=0$. This stays in contrast with our assumption \eqref{eq-j-at-zero}, so it would mean that 
there exists at least one touching point and we end up with the desired contradiction. 
\end{proof}

In the following proposition we show that in the treated setting, condition \eqref{eq-j-at-zero} is also sufficient for the nonlocal version of the Alexandrov Theorem, as it was already mentioned in the introduction.

\begin{proposition}\label{prop:counter}
Suppose that condition \eqref{eq-j-at-zero} does not hold, then there exists an open, connected and bounded set $E$ (with boundary of class $\mathcal{C}^\infty$) such that $H_j(\cdot, E)$ is constant along $\partial E$, but $E$ is not a ball.
\end{proposition}

\begin{proof}
We first notice that if condition \eqref{eq-j-at-zero} is not satisfied, then \eqref{eq:levy} is valid for any $\beta \in (0,2]$, as there exists $r>0$ such that $j$ is constant on the interval $(0,r]$. We next take any Borel set $E$ 
 such that $0<\mathrm{diam}(E)<r$. Then, for any $x\in \partial E$, it holds
\begin{align*}
H_j (x,E) 
&=
\lim_{\varepsilon \downarrow 0}
\int_{B_\varepsilon ^c(x)} \widetilde{\chi}_E(y)j(x-y)\, \ud y
=
\left(\int_{B_r(x)} +\int_{B_r^c(x)}\right) \widetilde{\chi}_E(y)j(x-y)\, \ud y
\\
&=
\left\{ |E^c \cap B_r(x)| - |E|\right\}j(r) + \int_{B_r^c(x)} j(x-y)\, \ud y\\
&=
\left\{ |B_r(0)| -2 |E|\right\}j(r) + \int_{B_r^c(0)}j(y)\, \ud y,
\end{align*}
and whence $H_j(\cdot ,E)$ is constant along $\partial E$ while the set $E$ clearly does not have to be a ball. 
\end{proof}

\bibliographystyle{abbrv}
\bibliography{bib}

\end{document}